\documentclass[11pt,a4paper,twoside]{article}
\usepackage{amsmath,amsfonts,amssymb,amsthm,bbm,latexsym,mathrsfs}
\usepackage{graphicx,color,epsfig,fancyhdr,dsfont}
\usepackage{enumerate}
\usepackage{hyperref}
\usepackage{indentfirst}
\usepackage[all]{hypcap}
\usepackage[affil-it]{authblk}
\allowdisplaybreaks
\usepackage[text={15.6cm,21.6cm},top=10mm]{geometry}
\title{New inverse and implicit function theorems for differentiable maps with isolated critical points}

\author[$\dag$]{Liangpan Li}
	\affil[$\dag$]{School of Mathematics, Shandong University, Jinan, Shandong, 250100, China}	

	\affil[ ]{liliangpan@gmail.com}

\date{}

\newtheorem{theorem}{Theorem}[section]
\newtheorem{lemma}[theorem]{Lemma}

\newtheorem{prop}[theorem]{Proposition}

\theoremstyle{definition}

\newtheorem{example}[theorem]{Example}

\RequirePackage{amsmath}
\RequirePackage{amssymb}
\RequirePackage{amsthm}
 \RequirePackage{epsfig}

\allowdisplaybreaks
\numberwithin{equation}{section}

\begin{document}

\maketitle

\begin{abstract}

The central purpose of this article is to establish new
inverse and implicit function theorems for differentiable maps with isolated critical points.
One of the key ingredients is a discovery of the fact that differentiable maps with isolated critical points are discrete maps,
which means that algebraic topology methods could then be deployed to
explore relevant questions.
We also provide a purely topological version of  implicit function theorem for continuous maps
that still possesses unique existence and continuity.
All new results of the paper are optimal with respect to the choice of dimensions.

\medskip

%\vskip25pt
\end{abstract}

\pagestyle{fancy}
\fancyhf{}
\fancyhead[LE,RO]{\thepage}
\fancyhead[LO]{\small{Differentiable maps with isolated critical points}}
\fancyhead[RE]{\small{L. Li}}

\section{Introduction}

The first rigorous proof of the scalar-valued implicit function theorem (ImFT)
was given by Ulisse Dini
at the University of Pisa in around 1877  \cite{Dini,Don,Hu,Krantz,Scarpello}.
His method was then extended by mathematical induction
to the full version of vector-valued maps \cite{Ar,Fi,Goursat1905,Oliveira,Zorich}.
Using the idea of successive approximation, nowadays regarded
 as a prototype of  the contraction mapping principle in functional analysis,
Edouard Goursat \cite{Goursat} gave a succinct  proof of  the ImFT in 1903 (see also \cite{Brooks,Zhang}).
Without making iterative  induction on dimensions, Gilbert A. Bliss
 (1876--1951, \cite{Graves,Mc}) \cite{Bliss}
provided a straightforward calculus-level proof of the ImFT in 1912.
Unfortunately, such a beautiful work has rarely been seriously discussed
over the past 100 years.
Alternatively,  the classical ImFT can be deduced as an application from the classical inverse function theorem (InFT)
 \cite{Blackadar,Coleman,Don,Garling,Golu,Hirsch,Krantz,Munkres,Rudin,Spivak,Tao,Taylor}.
 We also refer the interested reader to \cite{Blackadar15,Don,Krantz,Oliveira16,Oliveira18,Pugh} for other proofs of the classical ImFT and its generalizations.

The central purpose of this article is to
establish new
  inverse and implicit  function theorems for differentiable maps with isolated critical points.

Recall first a classical result proved independently
 S. Radulescu and
M. Radulescu in 1989 \cite{Ra}, and J. Saint Raymond in 2002 \cite{SRaymond} (see also \cite{Tao11} by T. Tao in 2011).

\begin{theorem}\label{strengthenedInFT} Let $\Omega$ be an open subset of $\mathbb{R}^n$, and let $H:\Omega\rightarrow\mathbb{R}^n$
be a differentiable map without critical points. Then $H$ is a local diffeomorphism.
\end{theorem}

In the case of $H$ having  isolated critical points, we will achieve a quite
 useful result which ensures that algebraic topology methods could then be deployed to further explore
 the local behaviour of $H$ at these points.

\begin{theorem}\label{Lemma35}
Let $\Omega$ be an open subset of $\mathbb{R}^n$, and let $H:\Omega\rightarrow\mathbb{R}^n$
be a differentiable map. If $x_0\in\Omega$
is an isolated critical point of $H$, then there exists an open neighborhood of $x_0$ on which $H$
assumes the value $H(x_0)$ only at $x_0$.\end{theorem}

Let us recall a concept \cite{Biasi,Gran}: a map $H$ between two topological spaces is said to be discrete at
a given point $x_0$ if there exists an open neighborhood of $x_0$ on which $H$ assumes the value
$H(x_0)$ only at $x_0$. Theorem \ref{Lemma35} then means that Euclidean vector fields
with isolated critical points are discrete maps.

Based on Theorem \ref{Lemma35} and some other properties
such as simply-connectedness of spheres, we will establish an inverse functions theorem as follows.

\begin{theorem}\label{THM12}
Let $\Omega$ be an open subset of $\mathbb{R}^m$ ($m\geq3$), and let $H:\Omega\rightarrow\mathbb{R}^m$
be a differentiable map with at most finitely many critical points.
Then $H$ is a local homeomorphism.
\end{theorem}

The examples $x\mapsto x^2$ $(x\in\mathbb{R})$ and $z\mapsto z^2$ $(z\in\mathbb{C})$
show that this result is optimal with respect to the choice of dimensions.
 Interestingly enough, Theorem \ref{THM12}
can be improved infinitely many times and we shall propose a potential limit at the end of Section \ref{section3}.

Next, we move on to establish implicit function theorems.

\begin{theorem}\label{THM11}
Let $\Omega$ be an open subset of $\mathbb{R}^{n}\times\mathbb{R}^m$, $n\geq1$, $m\geq2$, and
let $F=F(x,y)$
be a continuous map from  $\Omega$   to $\mathbb{R}^m$.
Assume $(x_0,y_0)$ is a point of $\Omega$ such that
the restriction of $F$ onto the section
\[\Omega_{x_0}\doteq\{y\in\mathbb{R}^m:(x_0,y)\in\Omega\}\]
is a differentiable map having $y_0$ as one of its isolated critical points.
Then  there exists an open neighbourhood $U\times V\subset \Omega$
of $(x_0,y_0)$
and a map $g:U\rightarrow V$  such that $F(x,g(x))=F(x_0,y_0)$
for all $x\in U$. Moreover, $g$ is continuous at $x_0$ as long as $V$ is  sufficiently small.
\end{theorem}

The polynomial
 \[F:(x,y)=(\left(\begin{array}{c}x_1\\ \vdots\\ x_n
 \end{array}
 \right),y)\mapsto (x_1-1)^2+x_2^2+\cdots+x_n^2+y^2-1\]
at the origin of $\mathbb{R}^n\times\mathbb{R}$ shows that the above ImFT
 is also optimal with respect to the choice of dimensions. We remark that
Theorem \ref{THM11}  extends Alexander and Yorke's result \cite{Al} from the usual setting of
regular points to isolated critical points, both on a single section $\Omega_{x_0}$, when the dimension $m$
is higher than one.

It would be a little bit unpleasant to see, as we will
 provide such examples later on, that the continuity of possible implicit functions derived from
Theorem \ref{THM11} can only be guaranteed at the chosen point $x_0$.
Our next result provides a purely topological ImFT which ensures both local
existence and continuity (see also \cite{Biasi}
for a relevant result).

\begin{theorem}\label{THMtopology}
Let $\Omega$ be an open subset of $\mathbb{R}^{n}\times\mathbb{R}^m$, and
let $F=F(x,y)$
be a continuous map from  $\Omega$   to $\mathbb{R}^m$.
Suppose  for each $x\in\mathbb{R}^n$,
the restriction of $F$ onto the $y$-section
$\Omega_{x}$
is a local homeomorphism.
Then  for any $(x_0,y_0)\in\Omega$, there exists an open neighbourhood $U\times V\subset \Omega$
of $(x_0,y_0)$
and a continuous map $g:U\rightarrow V$  such that for any $(x,y)\in U\times V$, $F(x,y)=F(x_0,y_0)$ if and only if $y=g(x)$.
\end{theorem}

A combination of Theorems \ref{strengthenedInFT} and
 \ref{THMtopology} yields the following ImFT  proved independently by
L. Hurwicz and M. K. Richter  in 1994  \cite{Hu}, and J. Saint Raymond in 2002 \cite{SRaymond}.

\begin{theorem}\label{thm27}
Let $\Omega$ be an open neighbourhood of a point $(x_0,y_0)$ of $\mathbb{R}^{n}\times\mathbb{R}^m$, and
let $F:\Omega\rightarrow\mathbb{R}^m$
be a differentiable map
such that $\frac{\partial F}{\partial y}$ is everywhere invertible.
Then there exists an open neighbourhood $U\times V\subset \Omega$
of $(x_0,y_0)$ and a differentiable map $g:U\rightarrow V$ such that for any $(x,y)\in U\times V$,
$F(x,y)=F(x_0,y_0)$ if and only if $y=g(x).$
\end{theorem}

There is a tiny issue concerning the differentiability of the implicit function claimed by Theorem \ref{thm27},
which,  not covered by Theorems \ref{strengthenedInFT} and
 \ref{THMtopology},  is left as a simple exercise.

The whole article would not make much sense if we are unaware of some examples of
  higher-dimensional differentiable Euclidean vector fields that possess isolated
critical points. Our final result is an attempt of answering this issue broadly.

\begin{theorem}\label{THM17}
Let $E$ be an arbitrary closed subset of a sphere or  hyperplane in $\mathbb{R}^n$ ($n\geq 2$).
Then there exists a smooth map from $\mathbb{R}^n$ to $\mathbb{R}^n$ whose set of critical points is exactly $E$.
\end{theorem}

\textbf{Notions and notations}

\begin{itemize}
  \item A local diffeomorphism ($C^k$-diffeomorphism) is a local homeomorphism such that its local restrictions and local inverse maps are both differentiable ($k$-times continuously differentiable).

      \item The symbols $B$ and $\mathscr{B}$ are reserved  to denote open ball and closed ball, respectively.
     \item The closure (boundary) of a subset $E$ of a topological space is denoted by $\overline{E}$ ($\partial E$).
\end{itemize}

\section{Theorem \ref{Lemma35}: proof and generalization}

We first provide a proof of Theorem \ref{Lemma35}, which is  stated as Theorem \ref{T21} for convenience.

\begin{theorem}\label{T21}
Let $\Omega$ be an open subset of $\mathbb{R}^m$, and let $H:\Omega\rightarrow\mathbb{R}^m$
be a differentiable map. If $x_0\in\Omega$
is an isolated critical point of $H$, then there exists an open neighborhood of $x_0$ on which $H$
assumes the value $H(x_0)$ only at $x_0$.\end{theorem}

\begin{proof}
Without loss of generality, we may assume that $x_0=0$, $H(x_0)=0$, and there are no
other critical points of $H$ in $\Omega$ except $x_0$.
Pick an $s>0$ such that the closed ball $\mathscr{B}_s\doteq\{x\in\mathbb{R}^m:|x|\leq s\}$
is contained in $\Omega$. We claim that there exists an $s_0\in[\frac{s}{2},s]$
such that \[N\doteq\min_{|x|=s_0}|H(x)|>0.\]
Suppose this is not the case. For each $\lambda\in[\frac{s}{2},s]$,
one can find a zero-value point $x_{\lambda}$ of $H$ such that $|x_{\lambda}|=\lambda$.
Then picking an  accumulation point $w_0$ of the set
$\{x_{\lambda}:\lambda\in[\frac{s}{2},s]\}$ would yield a contradiction
to the local injectivity of $H$ near $w_0$ that is guaranteed by Theorem \ref{strengthenedInFT}. This proves the claim. For simplicity, it is of no harm  to assume $s=s_0$.
For each $k\in\mathbb{N}$, there are no critical points of $H$ in the annulus $\frac{s}{2^k}\leq|x|\leq\frac{s}{2^{k-1}}$,
on which  there thus are at most finitely many zero-value points of $H$.
Therefore, the set $\mathscr{F}$ of zero-value points of $H$
in $\mathscr{B}_s\backslash\{0\}$ is a finite  or countable set.
If $\mathscr{F}$ is finite, we then can conclude the proof as a sufficiently small neighbourhood of $x_0=0$
meets what the lemma needs. Hence to prove the lemma,
 it remains to show that $\mathscr{F}$ can not be a countable set.
 We argue by contradiction and suppose that $\mathscr{F}$ is formed of a sequence of distinct points
 $\{z_i\}_{i=1}^{\infty}$. Considering the intersection between $\mathscr{F}$
 and each annulus  $\frac{s}{2^k}\leq|x|\leq\frac{s}{2^{k-1}}$ is finite,
 one gets $z_i\rightarrow0$ as $i\rightarrow\infty$. In other words, the bounded set $\mathscr{F}$
 has exactly one accumulation point $x_0=0$.
 For each  $r\in(0,N)$, denote $B_r=\{x\in\mathbb{R}^m:|x|<r\}$ and
 the total number of  path-connected components of  $H^{-1}(B_r)$
 that contain at least one element of $\mathscr{F}$ by $X(r)$.
The path-connected component of $H^{-1}(B_r)$
that includes $x_0=0$ is an open set, hence it
 contains
 most elements of $\mathscr{F}$  that are
 close enough to $x_0=0$. To be precise, by most we mean exceptional points are finitely many.
 Consequently, the counting function $X$ on $(0,N)$ can only assume values in $\mathbb{N}$.
We further claim that it is both left and right  continuous.

\emph{Notations}: Ahead of proving the  claim, we introduce some notations about $X(r)$, where $r\in(0,N)$.
The $X(r)$ path-connected components of  $H^{-1}(B_r)$
 that contain at least one element of $\mathscr{F}$ are denoted by
 $\{G_i(r)\}_{i=1}^{X(r)}$. Clearly, every $G_i(r)$
 is contained in the interior of the closed ball $\mathscr{B}_s$.
 Denote $\mathscr{F}_i(r)=\mathscr{F}\cap G_i(r)$,
 the set of zero-value points of $H$ in $G_i(r)$.
 Without loss of generality, we may assume that $\mathscr{F}_1(r)$ contains
  most elements of $\mathscr{F}$. To be precise,
  we assume that there exists a $k(r)\in\mathbb{N}$ such that $\{z_i:i\geq k(r)\}\subset \mathscr{F}_1(r)$.
  The minimal distance between the pairwise disjoint components $\{G_i(r)\}_{i=1}^{X(r)}$
  is short for $d(r)$, which makes sense as long as $X(r)\geq2$.

 \emph{Proof of the left continuity}: First, let $i\in\{2,3,\ldots,X(r)\}$ be fixed. Note that $\mathscr{F}_i(r)$
 is a finite set  contained in the path-connected open set $G_i(r)$.
 So there exists a continuous curve $L_i(r):[0,1]\rightarrow G_i(r)$
 such that $\mathscr{F}_i(r)$ is covered by the range of $L_i(r)$, which is a path-connected compact set.
 Obviously, the maximal value of the restriction
 of $|H|$ onto the range of $L_i(r)$, denoted by $\tau(H,L_i(r))$,
 is less than $r$. Hence for any $\gamma$ between $\tau(H,L_i(r))$ and $r$,
all elements of $\mathscr{F}_i(r)$ have to stay in the same path-connected component of $H^{-1}(B_{\gamma})$.
We denote this component by
$\widetilde{G}_i(\gamma)$, which is a subset of $G_i(r)$.
Next, we focus on $i=1$.
Take an arbitrary closed ball $\mathscr{B}_{\epsilon}\subset G_1(r)$,
which must contain most elements of $\mathscr{F}_1(r)$;
then pick a continuous curve $L_1(r):[0,1]\rightarrow G_1(r)$ so that its range covers the remaining points of $\mathscr{F}_1(r)$
and $x_0=0$. The union between $\mathscr{B}_{\epsilon}$ and the range of $L_1(r)$
is a path-connected compact set, over which the maximal value of $|H|$, denoted by $\tau(H,L_1(r),\mathscr{B}_{\epsilon})$,
is less than $r$.  Hence for any $\gamma$ between $\tau(H,L_1(r),\mathscr{B}_{\epsilon})$ and $r$,
all elements of $\mathscr{F}_1(r)$ have to stay in the same path-connected component of $H^{-1}(B_{\gamma})$.
We denote this component by
$\widetilde{G}_1(\gamma)$, which is a subset of $G_1(r)$. To summarize, for any $\gamma$
less than and sufficiently close to $r$ and any $i\in\{1,2,\ldots,X(r)\}$,
$\widetilde{G}_i(\gamma)$ is a path-connected component of $H^{-1}(B_{\gamma})$ that satisfies
 $\mathscr{F}_i(r)\subset\widetilde{G}_i(\gamma)\subset G_i(r)$,
 from which one can easily deduce $X(\gamma)=X(r)$.
 Therefore, $X$ is left continuous at $r$.

  \emph{Proof of the right continuity}: First, we claim  that $d(r)$ is positive whenever it makes sense (that is, $X(r)\geq2$).
  We argue by contradiction and suppose $d(r)=0$. Then there exist two disjoint path-connected components
  $G_i(r)$, $G_j(r)$ of $H^{-1}(B_r)$ such that their distance is zero.
  Since both components are bounded, their boundaries have to intersect somewhere,
   say for example at some $x^{\star}\in\mathscr{B}_s$.
  Obviously, $|H(x^{\star})|=r$, which implies that $x^{\star}$ is distinct from $x_0=0$. For simplicity, denote $y^{\star}=H(x^{\star})$.
  According to Theorem \ref{strengthenedInFT},
  there exists an open neighbourhood $\mathcal{A}$ of $x^{\star}$, $\mathcal{A}\subset\mathscr{B}_s\backslash\{0\}$,
  and an open ball  $\mathcal{B}$ centered at $y^{\star}$,
  $\mathcal{B}\subset B_N$, such that $H|_{\mathcal{A}}:\mathcal{A}\rightarrow\mathcal{B}$
  is a homeomorphism (actually a diffeomorphism). Since the intersection
  of any two Euclidean balls is path-connected,
  so are $\mathcal{B}\cap B_r$ and its pre-image
  $\mathcal{A}\cap H^{-1}(B_r)$.
  To be clear,
  \[(H|_{\mathcal{A}})^{-1}(\mathcal{B}\cap B_r)=\mathcal{A}\cap(H|_{\mathcal{A}})^{-1}(B_r)=\mathcal{A}\cap H^{-1}(B_r).\]
   Since $x^{\star}$
  is an interior point of $
  \mathcal{A}$ as well as a boundary point of $G_i(r)$,
  one can focus on a sufficiently small open neighbourhood of $x^{\star}$ to pick an element $x^{\bullet}\in\mathcal{A}\cap G_i(r)$.
  Similarly, there exists an $x^{\circ}\in\mathcal{A}\cap G_j(r)$.
  Note that $\mathcal{A}\cap G_i(r)$ and $\mathcal{A}\cap G_j(r)$
  are contained in the path-connected set  $\mathcal{A}\cap H^{-1}(B_r)$,
hence $x^{\bullet}$ and $x^{\circ}$ can be connected by a continuous curve in $H^{-1}(B_r)$.
 This yields a contraction to the facts $x^{\bullet}\in G_i(r)$, $x^{\circ}\in G_j(r)$,
 and $G_i(r), G_j(r)$ are distinct path-connected components of $H^{-1}(B_r)$.
 To summarize so far, we have proved that $d(r)>0$ as long as $X(r)\geq2$.
 Next, we argue by contradiction that $X$ is right continuous at $r$.
 Suppose this is not the case, then there exist two disjoint path-connected components  of $H^{-1}(B_r)$,
which are of the form  $G_i(r)$ and $G_j(r)$, $1\leq i<j\leq X(r)$, such that
 for any $t\in(r,N)$, both components are contained in the same path-connected
 component of $H^{-1}(B_t)$.
We  arbitrarily fix two points $x^{\oplus}\in G_i(r)$ and
  $x^{\odot}\in G_j(r)$. Then for  any $\epsilon>0$,
 there exist a
 continuous curve $L_{\epsilon}:[0,1]\rightarrow\mathbb{R}^m$ connecting $x^{\oplus}$ and $x^{\odot}$
 such that
 \begin{equation}
 \label{CON}\max_{t\in[0,1]}|H(L_{\epsilon}(t))|\leq r+\epsilon.\end{equation}
 Applying Theorem \ref{strengthenedInFT} to $H$ at each $x\in\partial G_i(r)$
 yields a homeomorphism $H|_{\mathcal{A}_x}:\mathcal{A}_x\rightarrow \mathcal{B}_x$,
 $\mathcal{A}_x\subset\mathscr{B}_s\backslash\{0\}$
 an open neighbourhood of $x$,
  $\mathcal{B}_x\subset B_N$
 an open ball centered at $H(x)$.  We further claim that for any  $x\in\partial G_i(r)$,
\begin{align}
(H|_{\mathcal{A}_x})^{-1}(\mathcal{B}_x\cap B_r)&=\mathcal{A}_x\cap G_i(r),\label{H31}\\
(H|_{\mathcal{A}_x})^{-1}(\mathcal{B}_x\cap \partial B_r)&=\mathcal{A}_x\cap \partial G_i(r),\label{H32}\\
(H|_{\mathcal{A}_x})^{-1}(\mathcal{B}_x\backslash \overline{B_r})&=\mathcal{A}_x\backslash \overline{G_i(r)}.\label{H33}
\end{align}
To prove (\ref{H31}), note first
   \[(H|_{\mathcal{A}_x})^{-1}(\mathcal{B}_x\cap B_r)=\mathcal{A}_x\cap(H|_{\mathcal{A}_x})^{-1}(B_r)=\mathcal{A}_x\cap H^{-1}(B_r),\]
   which implies that $\mathcal{A}_x\cap H^{-1}(B_r)$ is path-connected (due to $\mathcal{B}_x\cap B_r$ is path-connected).
  Considering $x$ is both an interior point of $\mathcal{A}_x$
 and a boundary point of $G_i(r)$, we see that $\mathcal{A}_x\cap G_i(r)$ is not an empty set.
Consequently,
$\mathcal{A}_x\cap (H^{-1}(B_r)\backslash G_i(r))$ must be an empty set; otherwise,
elements of  $\mathcal{A}_x\cap G_i(r)$ and $\mathcal{A}_x\cap (H^{-1}(B_r)\backslash G_i(r))$
can be connected by continuous curves in $\mathcal{A}_x\cap H^{-1}(B_r)$, which contradicts to the fact that
$G_i(r)$ is a path-connected component of $ H^{-1}(B_r)$. This
  observation suffices to establish (\ref{H31}).
   To prove (\ref{H32}), note first
   $H|_{\mathcal{A}_x}$ maps $\mathcal{A}_x\cap \partial G_i(r)$
   into $\mathcal{B}_x\cap \partial B_r$. We claim that
   it is impossible to find an element of $\mathcal{A}_x\backslash \overline{G_i(r)}$
   whose image under $H$
   lies in $\mathcal{B}_x\cap \partial B_r$; otherwise, applying Theorem \ref{strengthenedInFT}
   to $H$ at this point yields plenty of elements belonging to
  $(H|_{\mathcal{A}_x})^{-1}(\mathcal{B}_x\cap B_r)$ but not $\mathcal{A}_x\cap G_i(r)$, a contradiction to (\ref{H31}).
  Therefore,  $H|_{\mathcal{A}_x}$  maps $\mathcal{A}_x\cap \partial G_i(r)$
  bijectively onto $\mathcal{B}_x\cap \partial B_r$, which implies (\ref{H32}).
At last, (\ref{H33}) is a consequence of (\ref{H31}) and (\ref{H32}).
For each $x\in\partial G_i(r)$, pick an open neighbourhood $\mathcal{O}_x$ of $x$ such that
its diameter is less than $\frac{d(r)}{2}$ and its closure $\overline{\mathcal{O}_x}$ is contained in $\mathcal{A}_x$.
Considering $\partial G_i(r)$ is a compact set  covered by $\{\mathcal{O}_{x}\}_{x\in\partial G_i(r)}$, we can  pick
  finitely many $\{\mathcal{O}_{x_k}\}_{k=1}^n$
  that forms a  sub-cover of $\partial G_i(r)$.
  Introduce an open set
  \[\mathcal{W}=G_i(r)\cup\bigcup_{k=1}^n\mathcal{O}_{x_k}.\]
Obviously, any element of $\overline{G_i(r)}$
is not a boundary point of $\mathcal{W}$, so
\[\partial \mathcal{W}\subset \Big(\bigcup_{k=1}^n\partial \mathcal{O}_{x_k}\Big)\backslash \overline{G_i(r)}=
\bigcup_{k=1}^n\big((\partial \mathcal{O}_{x_k})\backslash \overline{G_i(r)}\big)\subset
\bigcup_{k=1}^n\big(\mathcal{A}_{x_k}\backslash \overline{G_i(r)}\big).\]
Considering the crucial fact (\ref{H33}), we see that the minimal value of
the continuous function $|H|\big|_{\partial \mathcal{W}}$
on the compact set $\partial \mathcal{W}$ is greater than $r$.
The diameter of each  $\mathcal{O}_{x_k}$ is less than $\frac{d(r)}{2}$
while the distance between $G_i(r)$ and $G_j(r)$
is not smaller than $d(r)$, hence
any continuous curve $L:[0,1]\rightarrow\mathbb{R}^m$ connecting $x^{\oplus}\in G_i(r)$
and $x^{\odot}\in G_j(r)$ must meet $\partial\mathcal{W}$ because it has to leave $ \mathcal{W}$. Therefore,
\[\max_{t\in[0,1]}|H(L(t))|\geq\min_{z\in\partial \mathcal{W}}|H(z)|>r,\]
which is a contradiction to (\ref{CON})
whenever $\epsilon$ is set to be sufficiently small.
This concludes the proof of the right continuity of  $X$.

\emph{Non-constant counting function}:
 Since  $X$ is a continuous function assuming values in $\mathbb{N}$,
 it has to be a constant.
 This contradicts to the  fact
$\lim_{r\rightarrow0}X(r)=\infty$
 whose proof is as follows. Take    finitely many
 elements $\{z_i\}_{i=1}^q$ ($q\in\mathbb{N}$) of $\mathscr{F}$;
 then apply Theorem \ref{strengthenedInFT}
 at these regular points of $H$ to get
 homeomorphisms $H|_{\mathcal{A}_i}: \mathcal{A}_{i}\rightarrow \mathcal{B}_i$,
 where $z_i\in\mathcal{A}_i$,  $\{\mathcal{A}_i\}_{i=1}^q$
can be assumed to be pairwise disjoint open sets contained in $\mathscr{B}_s\backslash\{0\}$;
and finally pick a small enough $r\in(0,N)$ such that \[\overline{B_{r}}\subset\bigcap_{i=1}^q\mathcal{B}_i.\]
For each $i\leq q$, the path-connected component of $H^{-1}(B_r)$ that contains $z_i$
is precisely $(H|_{\mathcal{A}_i})^{-1}(B_r)$, which is a subset of $\mathcal{A}_i$.
Note that $\{\mathcal{A}_i\}_{i=1}^q$ are pairwise disjoint, hence $X(r)\geq q$. Considering $q\in\mathbb{N}$
is arbitrary and $X$
is a non-increasing function, we obtain that $X(r)\rightarrow\infty$ as $r\rightarrow0$, and thus
finish the whole proof of Theorem \ref{T21}.
 \end{proof}

After carefully tracing back the role played by Theorem \ref{strengthenedInFT} in the proof of Theorem \ref{T21},
we find that it was only used to guarantee local homeomorphism of $H$
outside $x_0$. Thus we actually have a more general result as follows.

\begin{theorem}\label{THM42}
Let $\Omega$ be an open subset of $\mathbb{R}^m$, and let $H:\Omega\rightarrow\mathbb{R}^m$
be a continuous map such that $H$ is a local homeomorphism on $\Omega\backslash\{x_0\}$
for some $x_0\in\Omega$. Then there exists an open neighborhood of $x_0$ on which $H$
assumes the value $H(x_0)$ only at $x_0$
\end{theorem}

The role of the strict inequality \[\min_{x\in\partial\mathscr{B}_s}|H(x)|>0\]
played in the proof of Theorem \ref{T21} means that the boundary of $\mathscr{B}_s$
is a high wall for all zero-value points of $H$
in $\mathscr{B}_s$. Recall that  such a wall  is bound to exist locally surrounding the chosen point $x_0$.
In general, where one may detect high walls, where there could exist certain  uniqueness (or injectivity).

 \begin{theorem}\label{propA} Let $H:\overline{\Omega}\rightarrow\mathbb{R}^m$,
$\Omega$ a bounded  open subset of $\mathbb{R}^m$, be a continuous map
 whose restriction onto $\Omega$ is a local homeomorphism.
Suppose that \[N\doteq\min_{x\in\partial\Omega}|H(x)|>0.\]
Then there does not exist a continuous curve $L:[0,1]\rightarrow\Omega$ with distinct ending points $x_0$ and $x_1$
such that $H(x_0)=H(x_1)$ and the maximal modulus of $H$ over $L$ is less than $N$.
 \end{theorem}

A sketch of the proof is as follows. Suppose the contrary that
exist a continuous curve $L:[0,1]\rightarrow\Omega$ with distinct ending points $x_0$ and $x_1$
such that $H(x_0)=H(x_1)$ and the maximal modulus of $H$ over $L$ is less than $N$.
Without loss of generality we may assume that $H(x_0)=0$; otherwise
it suffices to study $\Psi\circ H:\overline{\Omega}\rightarrow\mathbb{R}^m$, where $\Psi$ is a self-homeomorphism of $\mathbb{R}^m$
mapping $H(x_0)$ to 0 and $B_N$ bijectively onto $B_N$.
For any $r\in(0,N)$, let $X(r)$ denote the total number
of path-connected components of $H^{-1}(B_r)$  that contain at least one element of $\{x_0,x_1\}$. Obviously, $1\leq X(r)\leq2$
for all $r$. Similar to the proof of Theorem \ref{T21}, one can show that $X(\cdot)$ is a constant-valued function on $(0,N)$.
Since $H$ is  homeomorphism locally around both $x_0$ and $x_1$, we see that $X(r)=2$ provided $r$ is sufficiently small.
Therefore, the counting function $X(\cdot)$ is identically equal to 2.
But considering  $x_0$ and $x_1$ are connected by $L$ over which
 the maximal modulus of $H$ is less than $N$, we get that $X(r)=1$ whenever $r$ is sufficiently close to $N$, a contradiction.

\section{Theorem \ref{THM12}: proof and self-improvement}\label{section3}

This main purpose of this section is to present a proof of Theorem \ref{THM12}, which stands in dimensions higher than two.
The corresponding  planar case
is slightly complicated and will  be discussed in detail in the next section.
We first collect three lemmas.

\begin{lemma}\label{THMopen2}
Let  $\Omega$ be an open subset of $\mathbb{R}^m$ ($m\geq2$), and  $H:\Omega\rightarrow\mathbb{R}^m$
 a differentiable map with at most finitely many critical points. Then $H$ is an open map.
\end{lemma}

This result was established  by the author  \cite{Li}.
We shall understand it from the viewpoint of maximal or minimal modulus. For example,
suppose further that $\Omega$
is bounded in $\mathbb{R}^m$ and $H$ admits a (unique) continuous extension to $\overline{\Omega}$.
Then it is impossible for $|H|$
to attain its maximal value inside $\Omega$. Alternatively, if $|H|$
 attains its minimal value inside $\Omega$, then this particular value can only be zero.

\begin{lemma}\label{SC}
Given an arbitrary simple closed curve $\gamma:[0,1]\rightarrow\mathbb{R}^2$, there exists a self-homeomorphism of
$\mathbb{R}^2$ under which the unit circle $\mathbb{S}$ is mapped onto the range of $\gamma$.
\end{lemma}

This result is known as the Sch\"{o}nflies theorem \cite{Am,Cai,Lee202,Mo}, which implies Jordan's curve theorem.

\begin{lemma}\label{HOM}
Let $X$ be path-connected and $Y$ simply connected Hausdorff spaces.
Then a local homeomorphism $f:X\rightarrow Y$ is a global homeomorphism from $X$ onto $Y$ if and only if  $f$ is proper.
\end{lemma}

To be clear, a continuous map from $X$ to $Y$ is said to be proper if the inverse image of each compact
subset of $Y$ is a compact subset of $X$. A proof of this result
was given by Ho \cite[Thm. 2]{Ho} (see also \cite{Gordon,Ho2}).
As an application of Lemma \ref{HOM}, local homeomorphisms
 between topological $(m-1)$-spheres ($m\geq3$)
 are bound to be global homeomorphisms.

We are now ready to establish Theorem \ref{THM12}, which is a consequence of Theorem \ref{strengthenedInFT}
and the following result.

\begin{theorem}\label{P48}
Let $m\geq2$, $\Omega$ be an open subset of $\mathbb{R}^m$,   $H:\Omega\rightarrow\mathbb{R}^m$
 a differentiable map,
$x_0\in\Omega$ an isolated critical point of $H$, and denote $y_0=H(x_0)$. If $r>0$ is sufficiently small,
then the path-connected component $V_r$ of $H^{-1}(B(y_0,r))$ that contains $x_0$
is a topological open $m$-ball.
In the case of $m\geq3$, $H|_{V_r}$ is a homeomorphism
from
$V_r$ onto $B(y_0,r)$.
\end{theorem}

\begin{proof} According to Theorem \ref{Lemma35}, we may assume without loss of generality
 that $H$ attains the value $y_0$ in $\Omega$
only art $x_0$. Similarly, as $x_0$ is known to be an isolated critical point of $H$,
it is of no harm to assume that there are no other critical points of $H$ in $\Omega$ except $x_0$.
We also require  $x_0=y_0=0$ for simplicity.
Since $x_0=0$ is an interior point of $\Omega$,  we can take an $s>0$ such that
$\overline{B_s}\subset \Omega$,
 where $B_s$ is short for $B(x_0,s)$, the open ball with radius $s$ and center $x_0=0$ in $\mathbb{R}^m$.
Considering that $y_0=H(x_0)=0$ is an interior point of the open subset $H(\Omega)$
of $\mathbb{R}^m$
 (due to Lemma \ref{THMopen2}) and the minimal value of $|H|$ over the compact set $\partial B_s$ is positive,  we can fix an $r>0$ such that
$\overline{B_r}\subset H(B_s)$ and
\begin{equation}\label{RS}
r<\min\limits_{x\in\partial B_s} |H(x)|.
\end{equation}
To be clear, $B_r$ should   be  understood as $B(y_0,r)$ although $y_0=0$ has been identified with $x_0=0$.
Let $\mathscr{X}$ denote the path-connected component of $H^{-1}(B_r)$
that contains $x_0=0$.   Obviously, the open set $\mathscr{X}$
is contained in the closed ball $\overline{B_s}$.
We claim that
\begin{itemize}
\item (a) $H(\partial\mathscr{X})\subset\partial B_r$,
\item (b) $\mathbb{R}^m\backslash\overline{\mathscr{X}}$ is a  path-connected unbounded open set,
  \item (c) $\partial\mathscr{X}$ is a compact  $(m-1)$-dimensional   topological manifold.
\end{itemize}
A proof of the claim is as follows.
\begin{itemize}
  \item Proof of (a): It follows from $H(\mathscr{X})\subset B_r$ that $H(\partial\mathscr{X})\subset \overline{B_r}$.
 Suppose there was a $z\in\partial\mathscr{X}$ such that $H(z)\in B_r$. Obviously, $z$ is distinct from $x_0$,
 the unique critical point of $H$ in $\Omega$. Thus according to Theorem \ref{strengthenedInFT},
  $H$
is a  homeomorphism between some open neighbourhood $U_0$ of $z$
and open neighbourhood $U_1$ of $H(z)$. Shrinking $U_1$ if necessary, we may further assume that
$U_1\subset B_r$. Consequently, $U_0\subset H^{-1}(B_r)$, which implies that $z$ is an interior point of $H^{-1}(B_r)$.
Hence $z$ is not a boundary point of $\mathscr{X}$ because $\mathscr{X}$ is a path-connected component of $H^{-1}(B_r)$.
A contradiction is thus obtained. So we get $H(\partial\mathscr{X})\subset\overline{B_r}\backslash B_r=\partial B_r$.

\item Proof of (b): $\mathbb{R}^m\backslash\overline{\mathscr{X}}$ is clearly an open set. Considering
$\mathbb{R}^m\backslash\overline{B_s}\subset \mathbb{R}^m\backslash\overline{\mathscr{X}}$, we see that
$\mathbb{R}^m\backslash\overline{\mathscr{X}}$ is an unbounded set.
Suppose the contrary that $\mathbb{R}^m\backslash\overline{\mathscr{X}}$
has at least two path-connected components, then one can pick one of them, denoted by $U_2$, such that
$U_2\subset \overline{B_s}$. This existence  is due to the fact that $\mathbb{R}^m\backslash\overline{B_s}$
is contained in some unique unbounded path-connected component of  $\mathbb{R}^m\backslash\overline{\mathscr{X}}$.
 Obviously, $\partial U_2$ is a non-empty set satisfying
\[\partial U_2\subset\partial(\mathbb{R}^m\backslash\overline{\mathscr{X}})=\partial\overline{\mathscr{X}}=\partial\mathscr{X}.\]
Fix an arbitrary element
$u\in\partial U_2$.  Since $u$ is distinct from $x_0$,
 the unique critical point of $H$ in $\Omega$,
one can deduce from Theorem \ref{strengthenedInFT} to get a homeomorphism from some open neighbourhood $U_3$ of $u$
onto some open ball $V_3$ centered at $H(u)\in\partial B_r$.
We claim that for any $w\in U_3$,
\[ w\in\mathscr{X}\Leftrightarrow |H(w)|<r.\]
Suppose this was not true, then there exists a $w_1\in U_3\backslash\mathscr{X}$ such that $|H(w_1)|<r$.
Take an arbitrary $w_2\in U_3\cap\mathscr{X}$ and note $|H(w_2)|<r$.
Thus $H(w_1)$ and $H(w_2)$ can be connected by a continuous curve in  $V_3\cap B_r$, which implies that
$w_1$ and $w_2$ can be connected by a continuous curve in the pre-image $U_3\cap H^{-1}(B_r)$, or simply in $H^{-1}(B_r)$.
So $w_1$ and $w_2$ have to stay in the same path-connected component of $H^{-1}(B_r)$. Considering
$w_2\in\mathscr{X}$, one gets $w_1\in\mathscr{X}$, a contradiction. This proves the  if-and-only-if proposition.
Since $u\in\partial U_2$ is an interior point of $U_3$,
one can pick a $w^{\star}\in U_2\cap U_3$. Note
$U_2\cap \overline{\mathscr{X}}=\emptyset$, so $w^{\star}\not\in\mathscr{X}$, which results in $|H(w^{\star})|\geq r$ (due to the if-and-only-if proposition).
Therefore, by considering  $H(\partial U_2)\subset\partial B_r$ and $\overline{U_2}$ is a compact set,
we see that \[\max_{y\in \overline{U_2}}|H(y)|=\sup_{y\in U_2}|H(y)|\] is attainable at some interior point of $U_2$,
which contradicts to the openness of $H(U_2)$ (due to Lemma \ref{THMopen2} or Theorem \ref{strengthenedInFT}).
Hence $\mathbb{R}^m\backslash\overline{\mathscr{X}}$  has exactly one path-connected component; in other words, it is a path-connected subset of $\mathbb{R}^m$.

\item Proof of (c):
Since $\partial\mathscr{X}$ is a bounded closed subset of $\mathbb{R}^m$,
we see that it is a compact topological space (endowed with the subspace topology).
It is easy to deduce from Theorem \ref{strengthenedInFT}  that
$\partial\mathscr{X}$ is an $(m-1)$-dimensional topological manifold.

\end{itemize}

We continue to prove the theorem by considering $m=2$ and $m\geq3$ separately.

\textbf{Case 1}: Suppose $m=2$. According to property (c), $\partial\mathscr{X}$ is a compact  one-dimensional   topological manifold.
  So every path-connected component of $\partial\mathscr{X}$
  can be regarded as a simple closed curve. If $\partial\mathscr{X}$ was not path-connected,
  we then can find two disjoint simple closed curves $\gamma_1,\gamma_2:\mathbb{S}\rightarrow\mathbb{R}^2$
  whose images are contained in $\partial\mathscr{X}$.
  Note
  \[\mathscr{X}\cup(\mathbb{R}^2\backslash\overline{\mathscr{X}})=\mathbb{R}^2\backslash\partial\mathscr{X}\subset
  \mathbb{R}^2\backslash(\gamma_1(\mathbb{S})\cup\gamma_2(\mathbb{S})).\]
 Both $\mathscr{X}$ and $\mathbb{R}^2\backslash\overline{\mathscr{X}}$
are path-connected (due to property (b)), while $\mathbb{R}^2\backslash(\gamma_1(\mathbb{S})\cup\gamma_2(\mathbb{S}))$
has exactly three path-connected components (due to Jordan's curve theorem). Thus
 one of the three path-connected components of $\mathbb{R}^2\backslash(\gamma_1(\mathbb{S})\cup\gamma_2(\mathbb{S}))$
 does not meet $\mathscr{X}\cup(\mathbb{R}^2\backslash\overline{\mathscr{X}})$,
 which in turn implies that this component
 is contained in the complement of $\mathscr{X}\cup(\mathbb{R}^2\backslash\overline{\mathscr{X}})$. Note that
 \[\mathbb{R}^2\backslash(\mathscr{X}\cup(\mathbb{R}^2\backslash\overline{\mathscr{X}}))=
\mathbb{R}^2\backslash(\mathbb{R}^2\backslash\partial\mathscr{X})=\partial\mathscr{X}.\]
 Consequently, we have derived a contradiction since a two-dimensional non-empty open set
 cannot be contained in a one-dimensional compact topological manifold.
 Hence $\partial\mathscr{X}$ has exactly one path-connected component. In other words, $\partial\mathscr{X}$ is a simple closed curve.
According to the Sch\"{o}nflies theorem (see Lemma \ref{SC}),
$\mathscr{X}$, the interior part bounded by the curve $\partial\mathscr{X}$,
is a topological open 2-ball.

 \textbf{Case 2}: Suppose $m\geq3$.
   According to property (c), $\partial\mathscr{X}$ is a compact $(m-1)$-dimensional topological manifold,
  so is an arbitrary path-connected component $\mathscr{Y}$ of  $\partial\mathscr{X}$.
Obviously (see also the proofs of properties (a) $\sim$ (c)),
 \[H|_{\mathscr{Y}}:\mathscr{Y}\rightarrow\partial B_r\]
  is a local homeomorphism.
Considering $\mathscr{Y}$ is a compact topological space, we see that
$H|_{\mathscr{Y}}$ is a proper map.
According to Lemma \ref{HOM}, $H|_{\mathscr{Y}}$ is a (global) homeomorphism from ${\mathscr{Y}}$
onto $\partial B_r$. In general, every path-connected component of
$\partial\mathscr{X}$ is a topological $(m-1)$-sphere.
Similar to the proof in Case 1 (Jordan's curve theorem is replaced with the Jordan-Brouwer separation theorem),
$\partial\mathscr{X}$ can only have one path-connected component.
In other words, $\partial\mathscr{X}$ itself is path-connected.
To be more accurate, $\partial\mathscr{X}$ is a topological $(m-1)$-sphere.
To summarize so far, $H|_{\partial\mathscr{X}}$ is a  homeomorphism from $\partial\mathscr{X}$ onto $\partial B_r$.
Now we are ready to prove that $H|_{\mathscr{X}}$ is a  homeomorphism from $\mathscr{X}$
onto $B_r$.

\begin{itemize}
  \item Proof of surjectivity: Let $q\in(0,r)$ be arbitrary. Considering $H(\partial\mathscr{X})=\partial B_r$,
  we see that the boundary $\mathscr{Y}_q$ of the path-connected component of $H^{-1}(B_q)$ that contains $x_0=0$
  is a subset of $\mathscr{X}$. Similar to the previous analysis in Case 2,
  $\mathscr{Y}_q$ is a topological $(m-1)$-sphere and $H|_{\mathscr{Y}_q}$
  is a  homeomorphism from $\mathscr{Y}_q$ onto $\partial B_q$. In particular,
  $H(\mathscr{Y}_q)=\partial B_q$. This suffices to
establish $H(\mathscr{X})=B_r$ because $H(x_0)=0$ and
\[B_r\backslash\{0\}=\bigcup_{0<q<r}\partial B_q=\bigcup_{0<q<r}H(\mathscr{Y}_q)=H\Big(\bigcup_{0<q<r}\mathscr{Y}_q\Big)\subset H(\mathscr{X}).\]

  \item Proof of injectivity: Following the notations in the previous surjectivity part, we claim that
  it is impossible to find an element $x^{\bullet}\in\mathscr{X}\backslash\mathscr{Y}_q$
  such that $|H(x^{\bullet})|=q$. Suppose this was not the case. Then the path-connected component of $H^{-1}(\partial B_q)$
  that contains $x^{\bullet}$, denoted by $\mathscr{Y}_{x^{\bullet}}$, is easily seen to be a topological $(m-1)$-sphere
  staying inside $\mathscr{X}$ and differs from $\mathscr{Y}_q$. According to the Jordan-Brouwer separation theorem,
the two disjoint topological $(m-1)$-spheres  $\mathscr{Y}_{x^{\bullet}}$ and $\mathscr{Y}_q$
divide \[\mathbb{R}^m\backslash(\mathscr{Y}_{x^{\bullet}}\cup\mathscr{Y}_q)\] into three path-connected
components, which are
one unbounded  and two bounded open sets. Among these  components, let $\mathscr{W}$
denote the bounded one that does not contain $x_0=0$.
The continuous non-negative function
\[x\in\overline{\mathscr{W}}\mapsto |H(x)|\in\mathbb{R},\]
 defined on the compact set $\overline{\mathscr{W}}$,
 is clearly to be constant valued on $\partial\mathscr{W}$.
 So this function attains either maximal or minimal value inside $\mathscr{W}$.
According to Lemma \ref{THMopen2},
 $H(\mathscr{W})$ is an open subset of $\mathbb{R}^m$. Thus $|H|\big|_{\overline{\mathscr{W}}}$
 cannot assume maximal value inside $\mathscr{W}$.
Consequently, the minimal value of $|H|$ on $\overline{\mathscr{W}}$
  is bound to be attainable inside $\mathscr{W}$.
  Since $H(\mathscr{W})$ is an open subset of $\mathbb{R}^m$,
  this minimal value has to be zero. But $x_0$ is the only element of $\mathscr{X}$
  whose image under $H$ is zero. Therefore, $x_0\in\mathscr{W}$, a contradiction. This proves the injectivity claim.

  \item Proof of continuity: Obviously, $H|_{\mathscr{X}}$ is continuous on $\mathscr{X}$
  and its inverse map from $B_r$ to $\mathscr{X}$ is continuous on $B_r\backslash\{0\}$ (due to Theorem \ref{strengthenedInFT}).
  It remains to show that the inverse map is continuous at $y_0=0$. Let $\{y_k\}_{k=1}^{\infty}$
  be an arbitrary sequence of elements of $B_r$ such that $y_k$ tends to $y_0$.
  Recall that $\mathscr{X}$ is contained in the closed ball $\overline{B_s}$. So if
  $(H|_{\mathscr{X}})^{-1}$ was discontinuous at $y_0$, we then can find
   a convergent subsequence of $\{(H|_{\mathscr{X}})^{-1}(y_k)\}_{k=1}^{\infty}$, written as
   $\{(H|_{\mathscr{X}})^{-1}(y_{k_i})\}_{i=1}^{\infty}$,
   whose limit $x_{\infty}$ differs from $(H|_{\mathscr{X}})^{-1}(y_0)=x_0$. Since $H$ is continuous on $\Omega$, we get that
  \[H(x_{\infty})=\lim_{i\rightarrow\infty}H((H|_{\mathscr{X}})^{-1}(y_{k_i}))=\lim_{i\rightarrow\infty}y_{k_i}=y_0,\]
  which contradicts to the fact that $x_0$ is the only element of $\Omega$
whose image under $H$ is $y_0$. This proves the continuity claim.
\end{itemize}
The whole proof of Theorem \ref{P48} now is finished.
\end{proof}

%\begin{example}
%The self-homeomorphism $x\rightarrow |x|^2x$ of $\mathbb{R}^m$ having a single critical point at the origin
%is a typical application of Theorem \ref{THM12}.
%\end{example}

As claimed in the Introduction, we explain why  Theorem \ref{THM12}
can be improved again and again. Let $H:\Omega\rightarrow\mathbb{R}^m$ ($m\geq3$)
be a $C^1$  vector field whose critical points consist of
a convergent  sequence of distinct elements $\{x_k\}_{k=0}^{\infty}$ of $\Omega$ with limit $x_{0}$
(the interested reader may drop the $C^1$
condition and explore more general result). We claim that $H$ is a local homeomorphism on $\Omega$. According to Theorem \ref{THM12},
each critical point $x_k$, $k\geq1$, is a  topological regular point of $H$.
Therefore, $H$ is a local homeomorphism on $\Omega\backslash\{x_0\}$
and the role played by Theorem \ref{strengthenedInFT}
on the region $\Omega\backslash\{x_0\}$ in the proof of Theorem \ref{P48}
has gotten an ideal substitute. There are two more issues  we need to solve:
\begin{itemize}
  \item (i)  Is $H$ an open map?
   \item (ii) Is $H$ discrete at $x_0$?
\end{itemize}
Recall that $H$ is a local homeomorphism on $\Omega\backslash\{x_0\}$,
so it follows from Theorem \ref{THM42} that $H$ is discrete at $x_0$.
Regarding the other question, it suffices to note the following open mapping theorem of the author \cite{Li}:

\begin{lemma}\label{Lemma46}
Let  $\Omega$ be an open subset of $\mathbb{R}^m$ ($m\geq2$), and  $H:\Omega\rightarrow\mathbb{R}^m$
 a $C^1$ map with at most countably many critical points. Then $H$ is an open map.
\end{lemma}

Consequently, there exists an $s>0$ and  a (sufficiently small) $r>0$ such that
\begin{itemize}
  \item $\overline{B(y_0,r)}\subset H(B(x_0,s))$ where $y_0\doteq H(x_0)$,
  \item $r<\min\limits_{x\in\partial B(x_0,s)}|H(x)-y_0|$,
  \item $x_0$ is the only element of $\Omega$ (shrinking $\Omega$ if necessary) whose image under $H$ is $y_0$.
\end{itemize}
One then goes through the proof of Theorem \ref{P48} to see that  $H$ is a  homeomorphism on a small neighbourhood of $x_0$,
which confirms  that $H$ is a local homeomorphism  on $\Omega$.
 To summarize so far, we have derived the following improvement of Theorem \ref{THM12}.

\begin{theorem}\label{thm36}
Let $\Omega$ be an open subset of $\mathbb{R}^m$ ($m\geq3$), and
$H:\Omega\rightarrow\mathbb{R}^m$ a $C^1$ map
whose critical points have at most finitely many accumulation points in $\Omega$. Then $H$ is a local homeomorphism on $\Omega$.
\end{theorem}

Similar to the proof of the above theorem, one can improve this result again and again.
For example, the sufficient condition in Theorem \ref{thm36} can be relaxed to $\mbox{CP}^{(k)}\cap\Omega=\emptyset$
for some $k\in\mathbb{N}$ or even to $(\cap_{k=1}^{\infty}\mbox{CP}^{(k)})\cap\Omega=\emptyset$,
where $\mbox{CP}^{(k)}$ denotes the $k$-th derived set of the collection $\mbox{CP}$ of all critical points of $F$.
We end this section with a  question: can the the sufficient condition in Theorem \ref{thm36} be improved to
$\mbox{CP}$ being a countably infinite set?

%\begin{Conjecture}
%Let $\Omega$ be an open subset of $\mathbb{R}^m$ ($m\geq3$), and
%$H:\Omega\rightarrow\mathbb{R}^m$ a $C^1$ map
%with at most countably many critical points.  Then $H$ is a local homeomorphism on $\Omega$.
%\end{Conjecture}

\section{The planar case: winding numbers}\label{SEC4}

In this part we continue to
explore  the given map $H$ in Theorem \ref{P48} around its isolated critical point $x_0$ provided $m=2$.
Identify $\mathbb{S}$ with the unit circle $\{z\in\mathbb{C}:|z|=1\}$ of the complex plane,
and $\mathbb{R}^2$ with $\mathbb{C}$. Given any continuous map $g:\mathbb{S}\rightarrow\mathbb{S}$, there exists
a unique continuous function $\Upsilon(g):[0,1]\rightarrow\mathbb{R}$ such that
\begin{itemize}
  \item $\Upsilon(g)(0)\in[0,1)$,
  \item $g(e^{2\pi\mathrm{i}t})=e^{2\pi\mathrm{i}\Upsilon(g)(t)}$ for all $t\in[0,1]$,
   \item $\mathscr{N}(g)\doteq\Upsilon(g)(1)-\Upsilon(g)(0)\in\mathbb{Z}$.
\end{itemize}
For example, letting $g(z)=z^n$ for some $n\in\mathbb{Z}$, one  can determine $\Upsilon(g_n)(t)=nt$ for all $t\in[0,1]$.
It is well known that  $g$ is null homotopic if and only if  $\mathscr{N}(g)$,
called the winding number of $g$,
is equal to  zero. To be clear, $g$ is said to be null homotopic if there exists
 a continuous map
 \[K:[0,1]\times\mathbb{S}\rightarrow\mathbb{S}\]
such that $K(0,x)=g(x)$ for all $x\in\mathbb{S}$ and $K(1,\cdot)$ is constant-valued on $\mathbb{S}$.
 For instance,
if $g$ is a local homeomorphism, then $\Upsilon(g)$
is easily seen to be  strictly increasing or decreasing, which  implies that $\mathscr{N}(g)$ is non-zero, or equivalently
$g$
is non-null homotopic. In general,  one can similarly define winding number
for continuous maps between simple closed curves (or loops). Note that such a concept depends on
the choice of the chosen traveling directions of  the  loops. With these facts available, we further study the planar case of
Theorem \ref{P48}.

\begin{theorem}\label{THMTWO}
Let $\Omega$ be an open subset of $\mathbb{R}^2$, and let $H:\Omega\rightarrow\mathbb{R}^2$ be a differentiable map.
If $x_0\in\Omega$ is an isolated critical point of $H$, then there exists an integer $\mathscr{N}_H(x_0)>0$
such that for any sufficiently small $q>0$, the boundary of the path-connected component $\mathscr{X}_q$ of $H^{-1}(B(y_0,q))$
that contains $x_0$, where $y_0\doteq H(x_0)$, is a simple closed curve from which $H$
is a local homeomorphism onto $\partial B(y_0,q)$ with winding number $\pm\mathscr{N}_H(x_0)$. Moreover,
$H$ is a homeomorphism from  $\mathscr{X}_q$ onto $B(y_0,q)$ if and only if $\mathscr{N}_H(x_0)=1$.
\end{theorem}

\begin{proof}
According to the proof of Theorem \ref{P48}, we may assume without loss of generality
that $x_0$ is the only critical point of $H$ in $\Omega$, $y_0=0$, and there exists
 an $r>0$
such that the path-connected component  of $H^{-1}(B(y_0,r))$ that contains $x_0$, denoted by $\mathscr{X}$,
 is a topological open 2-ball whose boundary $\partial\mathscr{X}$ is a simple closed curve in $\Omega$.
It follows from the Sch\"{o}nflies theorem that there exists a  homeomorphism $\Phi$
from $\mathbb{B}^2\doteq\{x\in\mathbb{R}^2:|x|\leq1\}$ onto $\overline{\mathscr{X}}$
whose restriction onto $\partial\mathbb{B}^2=\mathbb{S}$
is a  homeomorphism from $\mathbb{S}$ onto $\partial\mathscr{X}$.
Since $H|_{\partial\mathscr{X}}:\partial\mathscr{X}\rightarrow\partial B_r$ (see  the proof of Theorem \ref{P48}
for why $H(\partial\mathscr{X})\subset\partial B_r$) is a local homeomorphism, we see that the composition map
\[G:z\mapsto \frac{H(\Phi(z))}{r}\]
is  a local homeomorphism from $\mathbb{S}$ to $\mathbb{S}$.
Thus the winding number  $\mathscr{N}(G)$ of $G$ is non-zero,
which, combined with the definition of $\Upsilon(G)$, implies that $G$ is a surjective map onto $\mathbb{S}$. Therefore, $H(\partial\mathscr{X})=\partial B_r$.
Considering $\Upsilon(G)$ is strictly increasing or decreasing on $[0,1]$, we also have
\begin{equation}\sharp\{x\in\partial\mathscr{X}:H(x)=y\}=|\mathscr{N}(G)|\end{equation}
for all $y\in\partial B_r$. Based on this result, we further claim that
\begin{equation}\label{aim}\sharp\{x\in\mathscr{X}:H(x)=w\}=|\mathscr{N}(G)|\end{equation}
for all $w\in B_r\backslash\{0\}$. A sketch of proof is as follows.
Similar to the corresponding discussions in Case 2 of Theorem \ref{P48}
such as  denoting the boundary of the path-connected component of $H^{-1}(B_q)$ ($q\in(0,r]$) that contains $x_0=0$
still by $\mathscr{Y}_q$, it is not hard to establish that
\begin{itemize}
\item $\mathscr{Y}_q$ is a simple closed curve,
  \item $H(\mathscr{Y}_q)=\partial B_q$,
  \item $(H|_{\mathscr{X}})^{-1}(\partial B_q)=\mathscr{Y}_q$ (due to the open mapping property of $H$),
  \item $\sharp\{x\in\mathscr{Y}_q:H(x)=w\}$ is independent of $w\in\partial B_q$ (the common value  $\doteq\mathscr{N}_q$).
\end{itemize}
Thus prove (\ref{aim}), it suffices to show that $\mathscr{N}_q$ is independent of $q\in(0,r]$.
To this end, let $q^{\bullet}\in(0,r]$, $w^{\bullet}\in\partial B_{q^{\bullet}}$
be arbitrary, and denote
\[\{x\in\mathscr{Y}_{q^{\bullet}}:H(x)=w^{\bullet}\}=\{x_k^{\bullet}:k=1,2,\ldots,\mathscr{N}_{q^{\bullet}}\}.\]
Applying  standard compactness arguments to the compact subset $\mathscr{Y}_{q^{\bullet}}$ of $\mathbb{R}^2$
via Theorem \ref{strengthenedInFT}, one can find a $\gamma>0$, open neighbourhoods $U_{x_k^{\bullet}}$
of $x_k^{\bullet}$, and finitely many open subsets $\{U_{i}\}_{i=1}^{n}$ of $\mathbb{R}^2$ such that
\begin{itemize}
  \item (i) each $H|_{U_{x_k^{\bullet}}}$ is a homeomorphism from $U_{x_k^{\bullet}}$ onto $B(w^{\bullet},\gamma)$,
   \item (ii) $\{U_{x_k^{\bullet}}\}_{k=1}^{\mathscr{N}_{q^{\bullet}}}$ are pairwise disjoint,
  \item (iii) each $\overline{H(U_{i})}$ does not contain $w^{\bullet}$,
  \item (iv) $\{U_{x_k^{\bullet}}\}_{k=1}^{\mathscr{N}_{q^{\bullet}}}$ and  $\{U_{i}\}_{i=1}^{n}$ form an open cover of $\mathscr{Y}_{q^{\bullet}}$.
\end{itemize}
Suppose $w\in B_r\backslash\{0\}$ is sufficiently close to $w^{\bullet}$. Then
$(H|_{\mathscr{X}})^{-1}(\{w\})$, a subset of $\mathscr{Y}_{|w|}$, has to be covered by the union of $\{U_{x_k^{\bullet}}\}_{k=1}^{\mathscr{N}_{q^{\bullet}}}$ and  $\{U_{i}\}_{i=1}^{n}$ (due to (iv)) and stays away
from
 $\{U_{i}\}_{i=1}^{q}$ (due to (iii)). Therefore, $(H|_{\mathscr{X}})^{-1}(\{w\})$ is contained only in $\{U_{x_k^{\bullet}}\}_{k=1}^{\mathscr{N}_{q^{\bullet}}}$, which  combined with (i) and (ii) shows that
 $\mathscr{N}_{|w|}=\mathscr{N}_{q^{\bullet}}$. So we have proved that $\mathscr{N}_q$
a locally constant-valued function. This suffices to
establish (\ref{aim}). Because of this property we also write $|\mathscr{N}(G)|$
simply as $\mathscr{N}_H(x_0)$. To finish the proof of Theorem \ref{THMTWO}, it obviously suffices  to show that
$H$ is a homeomorphism from  $\mathscr{X}$ onto $B(y_0,r)$ if and only if $\mathscr{N}_H(x_0)=1$.
The necessity part is  evident; while as regard to the sufficiency part, one can  deduce from  $H(\mathscr{Y}_q)=\partial B_q$ for all $q\in(0,r)$ that $H$ is surjective, argue from (\ref{aim})
that $H|_{\mathscr{X}}$ is injective, which in turn implies that $(H|_{\mathscr{X}})^{-1}$ is continuous.
We are done.
\end{proof}

\begin{example}
Given $H(z)=z^n$ ($n\in\mathbb{N}$, $n\geq2$), it is easy to see that $\mathscr{N}_{H}(z=0)=n$.
Suppose $F(z)=z|z|^2$, then it is not hard to show that $\mathscr{N}_{F}(z=0)=1$.
\end{example}

\section{Proofs of Theorems \ref{THM11} and \ref{THMtopology}}

This section is devoted to proving Theorems \ref{THM11} and \ref{THMtopology}.
First, let us recall some basic knowledge about Brouwer's fixed point theorem (\cite{Park,Park2}),
which states that every continuous map
from \[\mathbb{B}^m\doteq\{x\in\mathbb{R}^m:|x|\leq1\}\]
to itself has at least one fixed point. This property is trivial when $m=1$, so from now on we assume $m\geq2$. Among plenty of equivalent
statements of this classical result \cite{Park,Park2}, we are particularly interested in
the so-called No Retraction Theorem:
\begin{itemize}
  \item the identify map on $\mathbb{S}^{m-1}=\partial\mathbb{B}^m$ is non-null homotopic.
\end{itemize}
To be clear,  a continuous map $g$ from $\mathbb{S}^{m-1}$ to $\mathbb{S}^{m-1}$ is said to be null homotopic if there exists
 a continuous map
 \begin{equation}\label{F50}K:[0,1]\times\mathbb{S}^{m-1}\rightarrow\mathbb{S}^{m-1}\end{equation}
such that $K(0,x)=g(x)$ for all $x\in\mathbb{S}^{m-1}$ and $K(1,\cdot)$ is constant-valued on $\mathbb{S}^{m-1}$ (see also Section \ref{SEC4}).
Equivalently, $g$ is null homotopic if and only if there exists a continuous map
 \begin{equation}\label{F520}\widetilde{K}:\mathbb{B}^m\rightarrow\mathbb{S}^{m-1}\end{equation} whose restriction onto $\partial\mathbb{B}^m=\mathbb{S}^{m-1}$
 is $g$ because $K$ and $\widetilde{K}$ are related to each other by
the relation
$K(t,x)=\widetilde{K}((1-t)x)$
for all $t\in[0,1]$ and $x\in\mathbb{S}^{m-1}$. Consequently,
if there is a continuous map $G:\mathbb{B}^m\rightarrow\mathbb{R}^m$
that sends $\partial\mathbb{B}^m$ into $\mathbb{R}^m\backslash\{0\}$
and if the map
\[x\in \mathbb{S}^{m-1}\mapsto\frac{G(x)}{|G(x)|}\in \mathbb{S}^{m-1}\]
is non-null homotopic, then one can find an  $x^{\bullet}$ in the interior of $\mathbb{B}^m$ such that $G(x^{\bullet})=0$; otherwise,
given $G$ maps $\mathbb{B}^m$ into $\mathbb{R}^m\backslash\{0\}$, one  derives a contradiction
from
\[\widetilde{K}:x\in \mathbb{B}^{m}\mapsto\frac{G(x)}{|G(x)|}\in \mathbb{S}^{m-1}\]
because the restriction of $\widetilde{K}$ onto $\partial\mathbb{B}^m=\mathbb{S}^{m-1}$  has been assumed to be non-null homotopic.
Thus we have derived an old result which can be found in \cite[Thm. 1.1.1]{Nirenberg} by Nirenberg.

\begin{prop}\label{prop51}
Let $G:\mathbb{B}^m\rightarrow\mathbb{R}^m$ ($m\geq2$) be a continuous map such that it sends $\partial\mathbb{B}^m$ into $\mathbb{R}^m\backslash\{0\}$
and the map
\[x\in \mathbb{S}^{m-1}\mapsto\frac{G(x)}{|G(x)|}\in \mathbb{S}^{m-1}\]
is non-null homotopic. Then there exists an  $x^{\bullet}$ in the interior of $\mathbb{B}^m$ such that $G(x^{\bullet})=0$.
\end{prop}

The power of this trivial proposition depends on which maps from $\mathbb{S}^{m-1}$ to $\mathbb{S}^{m-1}$ are already known to be
 non-null homotopic. A close examination of the  proof can actually give us more:
$\overline{B(0,r)}$ lies in the range of $G$ where $r=\min\limits_{x\in\partial\mathbb{B}^m}|G(x)|$.
Combining the No Retraction Theorem with this strengthened result, we obtain Lax's formulation \cite{Lax}
of the Intermediate Value Theorem for continuous maps which claims that any continuous map $G$ from $\mathbb{B}^m$
to $\mathbb{R}^m$ whose restriction onto $\partial\mathbb{B}^m$ is the identity map must satisfy $\mathbb{B}^m\subset G(\mathbb{B}^m)$.

In general, two continuous maps $f,g:\mathbb{S}^{m-1}\rightarrow\mathbb{S}^{m-1}$
are said to be homotopically equivalent if there exists a continuous map $K$ of the  form (\ref{F50}) such that $f=K(0,\cdot)$
and $g=K(1,\cdot)$. It is well known that homotopy is an equivalence relation.

We are ready to present a proof of Theorem \ref{THM11}, which is stated below as Theorem \ref{THM410}.

\begin{theorem}\label{THM410}
Let $\Omega$ be an open subset of $\mathbb{R}^{n}\times\mathbb{R}^m$, $n\geq1$, $m\geq2$, and
let $F=F(x,y)$
be a continuous map from  $\Omega$   to $\mathbb{R}^m$.
Assume $(x_0,y_0)$ is a point of $\Omega$ such that
the restriction of $F$ onto the section
\[\Omega_{x_0}\doteq\{y\in\mathbb{R}^m:(x_0,y)\in\Omega\}\]
is a differentiable map having $y_0$ as one of its isolated critical points.
Then  there exists an open neighbourhood $U\times V\subset \Omega$
of $(x_0,y_0)$
and a map $g:U\rightarrow V$  such that \[F(x,g(x))=F(x_0,y_0)\]
for all $x\in U$. Moreover, $g$ is continuous at $x=x_0$ as long as $V$ is  sufficiently small.
\end{theorem}

\begin{proof} We assume without loss of generality that $y_0=F(x_0,y_0)=0$, and first deal with the case $m\geq3$.
According to Theorem \ref{THM12},
there exists an $r>0$ such that $B(y_0,r)\subset\Omega_{x_0}$ and the map \[G:y\in B(y_0,r)\mapsto F(x_0,y)\in\mathbb{R}^m\]
is a homeomorphism from $B(y_0,r)$ onto some open neighbourhood $W$ of  the origin in $\mathbb{R}^m$.
Since $F$ is continuous at $(x_0,y_0)$ and $F(x_0,y_0)=0$ is an interior point of $W$,  we can pick a  positive constant $s$, $s<r$, such that the image of the  subset
 $\overline{B(x_0,s)}\times \overline{B(y_0,s)}$ of $\Omega$ under $F$ is contained in $W$.
Obviously,
\[\min_{|y|=s}|F(x_0,y)|>0.\]
So by considering the uniform continuity of $|F|$ on the compact set $\overline{B(x_0,s)}\times \overline{B(y_0,s)}$,
we can pick another positive constant $\gamma$, $\gamma< s$, such that
\begin{equation}\label{F51}
\min_{|x-x_0|\leq\gamma,\atop|y|=s}|F(x,y)|\geq\frac{1}{2}\min_{|y|=s}|F(x_0,y)|>0.
\end{equation}
Note that for any $x\in\mathbb{R}^n$ with $|x-x_0|\leq\gamma$, the map
\[K:(t,z)\in[0,1]\times\mathbb{S}^{m-1}\mapsto \frac{G^{-1}\big(F((1-t)x_0+tx,sz)\big)}{\big|G^{-1}\big(F((1-t)x_0+tx,sz)\big)\big|}\in\mathbb{S}^{m-1}\]
is a homotopy between  the identify map $K(0,\cdot)$ on $\mathbb{S}^{m-1}$ and $K(1,\cdot)$
that sends $z\in\mathbb{S}^{m-1}$
to
\[\frac{G^{-1}\big(F(x,sz)\big)}{\big|G^{-1}\big(F(x,sz)\big)\big|}\in\mathbb{S}^{m-1}.\]
Thus $K(1,\cdot)$ is non-null  homotopic. We then apply Proposition \ref{prop51} to the continuous map
\[z\in\mathbb{B}^m\mapsto G^{-1}\big(F(x,sz)\big)\in\mathbb{R}^m\]
to get an element $z^{\bullet}$ in the interior of $\mathbb{B}^m$ such that $G^{-1}(F(x,sz^{\bullet}))=0$, or equivalently
$F(x,sz^{\bullet})=0$. To conclude the proof of Case $m\geq3$, it suffices to set $U=B(x_0,\gamma)$ and $V=B(y_0,s)$.

The case $m=2$ can be dealt with in much the same way. We  first apply Theorem \ref{THMTWO} to
get a  topological open 2-ball $\mathscr{X}$,
which is a neighbourhood of $y_0=0$
  as well as a subset of $\Omega_{x_0}$, such that $F(x_0,\cdot)$ maps $\partial\mathscr{X}$
  onto $\partial B(y_0,q)$, $q$ some   positive constant, and the  winding number of this map
  is non-zero.
The Sch\"{o}nflies theorem then yields a homeomorphism $\widetilde{G}$ from $\mathbb{B}^2$ onto $\overline{\mathscr{X}}$
  which maps $\partial \mathbb{B}^2$ onto $\partial\mathscr{X}$. After that,
  apply Proposition \ref{prop51} to the continuous map
  \[z\in\mathbb{B}^2\mapsto F(x,\widetilde{G}(sz))\in\mathbb{R}^2\]
 provided $s$ is some positive constant such that $\overline{B(x_0,s)}\times \widetilde{G}(s\mathbb{B}^2)\subset\Omega$
 and $x$ lies in $\overline{B(x_0,\gamma)}$,
 where $\gamma<s$ is another sufficiently small positive constant that ensures
  \[\min_{|x-x_0|\leq\gamma,\atop |z|=1}|F(x,\widetilde{G}(sz))|>0.\]
 The only issue left is why
\[H:z\in\mathbb{S}\mapsto\frac{F(x_0,\widetilde{G}(sz))}{|F(x_0,\widetilde{G}(sz))|}\in\mathbb{S}\]
is non-null homotopic as replacing $x_0$ with an arbitrary $x\in \overline{B(x_0,\gamma)}$  would yield the same homotopy type.
To this end, note that $H$ is homotopically equivalent to
\[z\in\mathbb{S}\mapsto\frac{F(x_0,\widetilde{G}(z))}{|F(x_0,\widetilde{G}(z))|}\in\mathbb{S},\]
whose winding number
 is  the same as that of
$y\in\partial\mathscr{X}\mapsto F(x_0,y)\in \partial B(y_0,q)$
despite a possible $\pm1$ sign. Therefore, $H$ is non-null homotopic.
  To conclude the proof of Case $m=2$, we set once again $U=B(x_0,\gamma)$ and $V=B(y_0,s)$.

No matter $m\geq3$ or $m=2$, we are able to obtain an open neighbourhood $U\times V\subset \Omega$
of $(x_0,y_0)$
and a map $g:U\rightarrow V$  such that $F(x,g(x))=F(x_0,y_0)$
for all $x\in U$. To ensure that $g$ is continuous at $x=x_0$,
it suffices to argue by contradiction and assume in advance that $y_0$ is the only element of $\overline{V}$
whose image under $F(x_0,\cdot)$ is $F(x_0,y_0)$,
which is doable  because of Theorem \ref{Lemma35}. This finishes the whole proof of Theorem \ref{THM410}.
\end{proof}

\begin{example} Consider the polynomial
\[F(x,y)=(x-1)^2+y^2-1\]
on $\mathbb{R}\times\mathbb{R}$.
Since $\frac{\partial F}{\partial y}(0,0)=0$, the classical ImFT does not apply to $F$ at the origin of $\mathbb{R}^2$.
Note that if $x<0$, then
$F(x,y)\geq(x-1)^2-1>0=F(0,0).$ Thus it is indeed impossible to find, locally around the origin $(0,0)$,
 some solutions $(x,y)$ to the equation $F(x,y)=F(0,0)$
that are of the  form $(x,y(x))$, $x\in I$, $I$ a one-dimensional open
neighbourhood of $x=0$. Note $\frac{\partial F}{\partial y}(0,y)=2y$, so $y=0$ is an isolated critical point of $F(0,\cdot)$.
Therefore, Theorem \ref{THM410} is no longer true if the dimension $m$ of the $y$-variable is  one.
\end{example}

\begin{example}
Consider the polynomial
\[K(z,w)=z^2-w^4\]
on $\mathbb{C}^2$, which is identified with a map from $\mathbb{R}^2\times\mathbb{R}^2$ to $\mathbb{R}^2$
so that one may apply Theorem \ref{THM410}. Obviously,
$w=0$ is an isolated critical point of the restriction of $K$ onto $\{z=0\}\times\mathbb{C}$. So
one can apply Theorem \ref{THM410} to extract certain information of local solutions $(z,w)$ to the equation $K(z,w)=K(0,0)=0$
around the origin of $\mathbb{C}^2$.  Denote $z=re^{\mathrm{i}\theta}$ ($r\geq0$, $0\leq\theta<2\pi$) in polar coordinates.
Define $g(z)=\sqrt{r}e^{\mathrm{i}\frac{\theta}{2}}$ if $r\geq0$ is rational, and $g(z)=-\sqrt{r}e^{\mathrm{i}\frac{\theta}{2}}$
otherwise. Then $K(z,g(z))=0$ for all $z\in\mathbb{C}$ and $g$ is continuous merely at $z=0$.
This shows that the conclusion of Theorem \ref{THM410} concerning continuity cannot be improved anymore (at least for $n=m=2$).
Can we find an \emph{everywhere continuous}  solution map on some  open neighbourhood of $z=0$?
The answer is also NO because it is impossible to define $z\mapsto\sqrt{z}$ continuously
 on  non-simply connected open subsets of $\mathbb{C}\backslash\{0\}$.
\end{example}

We explicitly  state Theorem \ref{THMtopology}  below as Theorem \ref{thm55} for the sake of convenience.

\begin{theorem}\label{thm55}
Let $\Omega$ be an open subset of $\mathbb{R}^{n}\times\mathbb{R}^m$, and
let $F=F(x,y)$
be a continuous map from  $\Omega$   to $\mathbb{R}^m$.
Suppose  for each $x\in\mathbb{R}^n$,
the restriction of $F$ onto the $y$-section
$\Omega_{x}$
is a local homeomorphism.
Then  for any $(x_0,y_0)\in\Omega$, there exists an open neighbourhood $U\times V\subset \Omega$
of $(x_0,y_0)$
and a continuous map $g:U\rightarrow V$  such that for any $(x,y)\in U\times V$, $F(x,y)=F(x_0,y_0)\Leftrightarrow y=g(x)$.
\end{theorem}

\begin{proof}
The case $m=1$ is essentially well known  and thus left as an exercise. From now on
we may assume without loss of generality  that $m\geq2$, $y_0=F(x_0,y_0)=0$ and $\Omega$
is of the form $B(x_0,r)\times B(y_0,s)$ for some $r,s>0$. Since the restriction of $F$
onto $\Omega_{x_0}$ is a local homeomorphism, there exists a positive $s_1<s$ such that
\begin{equation}\label{LOWER0}G:y\in \overline{B(y_0,s_1)}\mapsto F(x_0,y)\in\mathbb{R}^m\end{equation}
is a homeomorphism from $\overline{B(y_0,s_1)}$ onto its image $F(x_0,\overline{B(y_0,s_1)})$ under $F(x_0,\cdot)$.
Hence by considering
$y_0=F(x_0,y_0)=0$, one easily gets
\[N\doteq\min_{|y|=s_1}|F(x_0,y)|>0.\]
Since $|F|$ is uniformly continuous on the compact subset $\overline{B(x_0,\frac{r}{2})}\times \partial\overline{B(y_0,s_1)}$
of $\mathbb{R}^{n}\times\mathbb{R}^m$,
one can find a positive $r_1\leq\frac{r}{2}$ such that
\begin{equation}\label{LOWER}\min_{|x-x_0|\leq r_1, \atop |y|=s_1}|F(x,y)|\geq\frac{N}{2}.\end{equation}
Note also $|F|$ is continuous at $(x_0,y_0)$ and $|F|(x_0,y_0)=0$,
so there exists a large enough integer $k\in\mathbb{N}$ such that
\begin{equation}\label{LOWER1}\min_{|x-x_0|\leq \frac{r_1}{k}, \atop |y|\leq\frac{s_1}{k}}|F(x,y)|\leq\frac{N}{4}.\end{equation}
Similar to the deduction of (\ref{LOWER}) by replacing  $s_1$ with $s_2\doteq\frac{s_1}{k}$,
we can get a positive $r_2\leq\frac{r_1}{k}$ such that
\begin{equation}\label{LOWER2}\min_{|x-x_0|\leq r_2, \atop |y|=s_2}|F(x,y)|>0.\end{equation}
With these preparations available, we set $U=B(x_0,r_2)$ and $V=B(y_0,r_2)$.
Given any $x\in U$, the continuous map
  $z\in\mathbb{B}^{m}\mapsto F(x,s_2z)\in\mathbb{R}^m$
  sends $\mathbb{S}^{m-1}$ into $\mathbb{R}^m\backslash\{0\}$ (due to (\ref{LOWER2})) and
  \[z\in\mathbb{S}^{m-1}\mapsto \frac{F(x,s_2z)}{|F(x,s_2z)|}\in\mathbb{S}^{m-1}\]
  is homotopically equivalent to the non-null homotopic map
 \[z\in\mathbb{S}^{m-1}\mapsto \frac{F(x_0,s_2z)}{|F(x_0,s_2z)|}\in\mathbb{S}^{m-1}\]
(due to (\ref{LOWER2}) and (\ref{LOWER0})). It then follows from Proposition \ref{prop51} that there exists
a $g(x)\in V$ such that $F(x,g(x))=0$. We claim that it is impossible to
find an $x^{\bullet}\in U$ and $y^{\bullet}\in V$ with $y^{\bullet}\neq g(x^{\bullet})$
such that $F(x^{\bullet},y^{\bullet})=0$. If not, then it follows from (\ref{LOWER1}) that
the maximal value of $|F(x^{\bullet},\cdot)|$ over the straight line connecting $g(x^{\bullet})$ and $y^{\bullet}$
in $V$ is not bigger than $\frac{N}{4}$, which contradicts to (\ref{LOWER}) by Theorem \ref{propA}.
The only issue left is why the local solution map $g$ is continuous, which is a standard consequence of the
uniqueness we have just confirmed (see also the end of the proof of Theorem \ref{THM410}), thus omitted.
This finishes the proof of Theorem \ref{thm55}.
\end{proof}

\section{Theorem \ref{THM17}: construction of critical points}

Our proof of Theorem \ref{THM17} relies on the following well-known application of partition of unity \cite[p. 17]{Golu}.

\begin{lemma}\label{lemma61} Given an arbitrary closed subset $E$ of $\mathbb{R}^n$, there exists a smooth non-negative function  $g:\mathbb{R}^n\rightarrow\mathbb{R}$
such that  $E=g^{-1}(\{0\})$.
\end{lemma}

 A proof of Theorem \ref{THM17} is as follows. Suppose first $E$ is of the form $\{0\}\times \widehat{E}$,
 where $\widehat{E}$ is a closed  subset of
 $\mathbb{R}^{n-1}$. According to Lemma \ref{lemma61},
 there exists a smooth non-negative function  $g:\mathbb{R}^{n-1}\rightarrow\mathbb{R}$
such that  $\widehat{E}=g^{-1}(\{0\})$. The determinant of the Jacobian matrix of
the map
\[H_g:(x,z)\in\mathbb{R}\times\mathbb{R}^{n-1}\mapsto (x^3+g(z)x,z)\in\mathbb{R}\times\mathbb{R}^{n-1}\]
is easily seen to be $3x^2+g(z)$, thus the set of critical points of $H_g$
is exactly $\{0\}\times\widehat{E}=E$. Suppose next  $E$ is a closed subset of $\mathbb{S}^{n-1}$.
 According to Lemma \ref{lemma61},
 there exists a smooth non-negative function  $g:\mathbb{R}^{n}\rightarrow\mathbb{R}$
such that  $E=g^{-1}(\{0\})$.
Direct computation via rotational symmetry shows that the determinant of the Jacobian matrix of
the map
\[H_{g}:re\mapsto (r^3-3r^2+3r+g(e)r^2)e\]
in polar coordinates $r\in[0,\infty)$ and $e\in\mathbb{S}^{n-1}$
is
\[(3r^2-6r+3+g(e)2r)\cdot(r^2-3r+3+g(e)r)^{n-1},\]
hence the set of critical points of $H_g$ is precisely $E$. Both of the examples are obviously sufficient to
establish Theorem \ref{THM17}.

\section{Further applications}

\subsection{Non-null homotopy}

Belitskii and Kerner \cite{BK} considered the  map
\[F:(x,y)=(x, \left(\begin{array}{c}y_1\\ y_2
 \end{array}
 \right))\mapsto  \left(\begin{array}{c} y_1^2+xy_1-x^3\\ y_2^2+xy_2-x^3
 \end{array}
 \right)\]
 from $\mathbb{R}\times\mathbb{R}^2$
to $\mathbb{R}^2$. Around the origin of $\mathbb{R}\times\mathbb{R}^2$, Theorem \ref{THM410}
 does not apply to
$F$
because $y=0$ is not an isolated critical point of the map $y\mapsto F(x=0,y)$ on $\mathbb{R}^2$.
We remark that, however, the  idea of
Prop.
\ref{prop51}
still works for this example. Letting $x\in(-\frac{1}{10},\frac{1}{10})\backslash\{0\}$ be  fixed dummy variable
and letting $y\in\mathbb{R}^2$
be such that $|y|=2|x|^2$, one gets
\begin{align*}
|F(x,y)|&\geq|x|\cdot|y|-\sqrt{y_1^4+y_2^4}-\sqrt{2}|x|^3\\
&\geq 2|x|^3-4|x|^4-\sqrt{2}|x|^3\\
&\geq\frac{|x|^3}{10}>0.
\end{align*}
In much the same way, one can show that $|F_t(x,y)|>0$ where $x,y$ given as before and  $F_t(x,y)$ $(t\in[0,1])$  defined by
\[F_t:(x,y)=(x, \left(\begin{array}{c}y_1\\ y_2
 \end{array}
 \right))\mapsto  \left(\begin{array}{c} ty_1^2+xy_1-tx^3\\ ty_2^2+xy_2-tx^3
 \end{array}
 \right).\]
Therefore, the continuous map $y\mapsto F(x,y)$
from the circle $|y|=2|x|^2$ into the range space $\mathbb{R}^2\backslash\{0\}$ is easily seen to be homotopically
equivalent to the dilation map $y\mapsto xy$
on the same circle. A suitable application of Prop. \ref{prop51}
together with the No Retraction Theorem yield an element $g(x)\in\mathbb{R}^2$ ($x\in(-\frac{1}{10},\frac{1}{10})\backslash\{0\}$),
$|g(x)|<2|x|^2$, such that $F(x,g(x))=0$.

\subsection{Hadamard's diffeomorphism theorem}

Compared with local injectivity that is guaranteed by the classical inverse function theorem,
many authors  including Jacques Hadamard \cite{Ga,Gordon,Gordon2,Gu,Hadamard1,Hadamard2,Nollet,Plastock,Ruzhansky,Xavier}
have been interested in establishing
global injectivity for differentiable maps.
With   Theorems \ref{strengthenedInFT} and  \ref{propA}  available,
 one can generalize
Hadamard's diffeomorphism theorem \cite{Hadamard1,Hadamard2} (see also \cite{Baj,Ga0,Ga,Gordon,Gordon2,Ka,Krantz,Palais,Plastock,Ruzhansky}) to

\begin{prop}\label{prop71}  A differentiable map $H:\mathbb{R}^m\rightarrow\mathbb{R}^m$
is a  diffeomorphism (from $\mathbb{R}^m$ onto $\mathbb{R}^m$) if and only if it has no critical points and $|H(x)|$ tends to infinity as $|x|\rightarrow\infty$.
\end{prop}

We first establish the sufficiency part.
Let $H:\mathbb{R}^m\rightarrow\mathbb{R}^m$
be a differentiable map such that it has no critical points and  $|H(x)|$ tends to infinity as $|x|\rightarrow\infty$.
The first and second assumptions imply, respectively, that the range of $H$ is both open (see e.g. Theorem \ref{strengthenedInFT},
\cite[Prop. 3]{Li}, \cite[Lemma 4]{SRaymond}, \cite[Cor. 4]{Tao11}) and closed, hence $H$ is a surjective map onto $\mathbb{R}^m$.
Letting $y,z$ be two distinct points of $\mathbb{R}^m$, we denote
\[N_{yz}\doteq\max_{t\in[0,1]}\big|H((1-t)y+tz)\big|,\]
then pick an $R>0$ such that
\[\min_{|x|=R}|H(x)|>N_{yz}.\]
According to Theorem \ref{strengthenedInFT}, the restriction of  $H$ onto the open ball $B(x=0,R)$ is a local diffeomorphism,
thus local homeomorphism as well. Finally, an application of  Theorem \ref{propA}
yields $H(x)\neq H(y)$. In other words, $H$ is an injective map.
Therefore,  $H:\mathbb{R}^m\rightarrow\mathbb{R}^m$ is a bijective map without critical points,
which combined with Theorem \ref{strengthenedInFT} shows that $H$ is a diffeomorphism from $\mathbb{R}^m$ onto $\mathbb{R}^m$.
The necessity part regarding why $|H(x)|\rightarrow\infty$ as $|x|\rightarrow\infty$
can be argued routinely by contradiction arguments, thus
details are omitted. This finishes the proof of Proposition \ref{prop71}.

\subsection{Non-existence of 3-dimensional fields over $\mathbb{R}$}
In the history people made great effort
to construct 3-dimensional fields over $\mathbb{R}$, but failed.
We refer the interested reader to \cite[Thm. 6.2.9]{Krantz} by Krantz and Parks
for an elegant proof why such attempts should  never be successful.
Based on Theorem \ref{THM12}, we  provide another approach
although it is essentially the same as that just mentioned in \cite{Krantz}.
Suppose there exists a multiplication operation $\circledast$ on $\mathbb{R}^3$
such that $(\mathbb{R}^3,+,\circledast)$ forms a field and
$(\lambda x)\circledast y=x\circledast(\lambda y)=\lambda(x\circledast y)$
for all $x,y\in\mathbb{R}^3$ and $\lambda\in\mathbb{R}$.
Direct computation shows that the directional derivative of the map
\[H:x\in\mathbb{R}^3\mapsto x\circledast x\in\mathbb{R}^3\]
at $x\in\mathbb{R}^3$ with respect to an arbitrary vector $e\in\mathbb{R}^3$ is  given by
\[\frac{\partial G}{\partial e}(x)=(2x)\circledast e.\]
Thus $G$ is differentiable at $x$ (see e.g. \cite[Thm. 6.2]{Munkres}, \cite[Thm. 2-8]{Spivak}, \cite[Thm. 17.3.8]{Tao})
and the corresponding Jacobian matrix $J_G(x)$ can be identified with the linear map
\[e\mapsto (2x)\circledast e\]
on $\mathbb{R}^3$. Obviously, $J_G(x)$ is invertible if and only if $x\neq0$. In other words,
$x=0$ is the only critical point of $G$.
According to Theorem \ref{THM12}, $H$ is a homeomorphism on an open neighbourhood of the origin of $\mathbb{R}^3$.
On the other hand, letting $\mu$ denote the multiplicative unit of $(\mathbb{R}^3,\circledast)$,
one has
$(\lambda \mu)\circledast(\lambda \mu)=\lambda^2 \mu $ for all $\lambda\in\mathbb{R}$.
So $H$ is not an injective map on any neighbourhood of the origin of $\mathbb{R}^3$, a contradiction.

\subsection{Open maps}

Lemma \ref{Lemma46} is the main result of the author's paper \cite{Li}, which will appear soon
in Real Analysis Exchange. Regarding the countable infinity assumption therein,  the anonymous referee for the article  suggested
that it may be improved to zero one-dimensional Hausdorff measure.
We  can confirm this prediction is indeed true.

\begin{prop}\label{THM72}
Let  $\Omega$ be an open subset of $\mathbb{R}^m$ ($m\geq2$), and let $H:\Omega\rightarrow\mathbb{R}^m$
be a $C^1$ map such that  the one-dimensional Hausdorff measure of its critical points is zero. Then $H$ is an open map.
\end{prop}

In a summer project supervised by the author, Lijie Fang (Shandong Univ.), Dazhou Liu (Chongqing Three Gorges Univ.)
and Yueyue Wu (Central South University)
gave a further generalization  as follows.

\begin{theorem}\label{THM73}
Let  $\Omega$ be an open subset of $\mathbb{R}^m$ ($m\geq2$), and let $F:\Omega\rightarrow\mathbb{R}^m$
be a $C^1$ map such that
 $\mathscr{H}_{m-1}(\mbox{CP})=0$ and
 $\mathscr{H}_1(\mbox{CP}\cap F^{-1}(\{y\}))=0$ for all critical values  $y$.
 Then $H$ is an open map.
\end{theorem}

To be clear, $\mathscr{H}_s$ ($s\geq0$) denotes the $s$-dimensional Hausdorff measure and $\mbox{CP}$
is short for the set of critical points of $F$.
A full proof of Theorem \ref{THM73} will be presented in due course.

\end{document}